\documentclass[12pt, a4paper, leqno]{article}
\usepackage[margin=2.4cm]{geometry}
\usepackage[utf8]{inputenc} %---latin2, cp1250
\usepackage{lmodern}
\usepackage[T1]{fontenc}
\usepackage[english]{babel}

\usepackage[runin]{abstract}
\usepackage{titling}

\usepackage{amsmath}
\usepackage{amsthm}
\usepackage{amsfonts}
\usepackage{amssymb}
\usepackage{mathtools}
\usepackage{clrscode}
\usepackage{enumitem}
\usepackage{mdwlist}

\usepackage{cite}

\abslabeldelim{.}
\newcommand{\keywordsname}{Key words and phrases}
\makeatletter
\newcommand{\keywords}[1]{%
\def\thekeywords{#1}%
\begin{@bstr@ctlist}
\hspace*{\abstitleskip}{\abstractnamefont\keywordsname\@bslabeldelim}\abstracttextfont\
#1%
\par\end{@bstr@ctlist}
}
\makeatother
\newcommand{\subjclassname}{Mathematics subject classification}
\makeatletter
\newcommand{\subjclass}[2][2020]{%
\begin{@bstr@ctlist}
\hspace*{\abstitleskip}{\abstractnamefont\subjclassname\ (#1)\@bslabeldelim}\abstracttextfont\
#2%
\par\end{@bstr@ctlist}
}
\makeatother
\makeatletter
\def\and{%				%begin{tabular}
	\end{tabular}%
	and%
	\begin{tabular}[t]{c}}%
\makeatother
\makeatletter
\let\addresses\@empty      %\let\thankses\@empty
\newcommand{\address}[2][]{\g@addto@macro\addresses{\address{#1}{#2}}}
\newcommand{\curraddr}[2][]{\g@addto@macro\addresses{\curraddr{#1}{#2}}}
\newcommand{\email}[2][]{\g@addto@macro\addresses{\email{#1}{#2}}}
\newcommand{\urladdr}[2][]{\g@addto@macro\addresses{\urladdr{#1}{#2}}}
\def\enddoc@text{%\ifx\@empty\@translators \else\@settranslators\fi
  \ifx\@empty\addresses \else\@setaddresses\fi}
\AtEndDocument{\enddoc@text}
\def\emailaddrname{E-mail address}
\def\@setaddresses{\par
  \nobreak \begingroup
  \interlinepenalty\@M
  \def\address##1##2{\begingroup%
    \par\addvspace\bigskipamount%\indent
    \@ifnotempty{##1}{(\ignorespaces##1\unskip) }%
    {\noindent\ignorespaces##2}\par\endgroup}%
  \def\email##1##2{\begingroup
    \@ifnotempty{##2}{\nobreak\noindent\emailaddrname
      \@ifnotempty{##1}{, \ignorespaces##1\unskip}\/:\space
      \ttfamily##2\par}\endgroup}%
  \addresses
  \endgroup
}

\makeatother
\makeatletter
\def\cstar#1{\expandafter\@cstar\csname c@#1\endcsname}
\def\@cstar#1{\ifcase#1\or $\ast$\or $\ast\ast$\or $\ast\ast\ast$\fi}
\AddEnumerateCounter{\cstar}{\@cstar}{$\ast\ast\ast$}
\makeatother

\newlist{conditions}{enumerate}{1}
\newlist{exconditions}{enumerate}{1}
\newlist{iconditions}{enumerate}{1}
\newlist{inthm}{enumerate}{1}
\setlist[conditions]{label=\normalfont(\alph*),ref=\normalfont(\alph*)}
\setlist[exconditions]{label=\normalfont(\roman*),ref=\normalfont(\roman*),wide,labelindent=0pt}
\setlist[iconditions]{label=\normalfont(\roman*),ref=\normalfont(\roman*)}
\setlist[inthm]{label=\normalfont(\thetheorem.\arabic*),ref=\normalfont(\thetheorem.\arabic*),wide,labelindent=0pt}

\mathchardef\mhyphen="2D
\newcommand{\CB}{\mathbb{C}}
\newcommand{\G}{\mathbb{G}}
\newcommand{\PB}{\mathbb{P}}
\newcommand{\R}{\mathbb{R}}
\newcommand{\SB}{\mathbb{S}}
\newcommand{\Z}{\mathbb{Z}}

\newcommand{\C}{\mathcal{C}}
\newcommand{\UC}{\mathcal{U}}

\newcommand{\Cinfty}{\C^{\infty}}
\newcommand{\GL}{\mathrm{GL}}
\newcommand{\Halg}{H_{\mathrm{alg}}}
\newcommand{\Ker}{\operatorname{Ker}}
\newcommand{\SO}{\mathrm{SO}}

\newtheorem{theorem}{Theorem}[section]
\newtheorem{corollary}[theorem]{Corollary}
\newtheorem{proposition}[theorem]{Proposition}
\newtheorem{lemma}[theorem]{Lemma}
\theoremstyle{definition}
\newtheorem*{acknowledgements}{Acknowledgements}
\newtheorem{definition}[theorem]{Definition}
\newtheorem{example}[theorem]{Example}
\newtheorem{notation}[theorem]{Notation}
\newtheorem{problem}{Problem}

\newtheorem{remark}[theorem]{Remark}

\title{Rational approximation of holomorphic maps}
\date{}
\author{Jacek Bochnak \and Wojciech Kucharz}

\address{Jacek Bochnak\\Le Pont de l'\'Etang 8\\1323
Romainm\^otier\\Switzerland}
\email{jack3137@gmail.com}

\address{Wojciech Kucharz\\Institute of Mathematics\\Faculty of Mathematics and Computer
Science\\Jagiellonian University\\\L{}ojasiewicza 6\\30-348
Krak\'ow\\Poland}
\email{Wojciech.Kucharz@im.uj.edu.pl}

\usepackage[pdftex, pdfauthor={J. Bochnak, W. Kucharz},pdftitle={\thetitle}]{hyperref}

\begin{document}
\maketitle
\thispagestyle{empty}

\begin{abstract}
Let $X$ be a complex nonsingular affine algebraic variety, $K$ a compact holomorphically convex subset of $X$, and $Y$ a homogeneous variety for some complex linear algebraic group. We prove that a holomorphic map $f 
\colon K \to Y$ can be uniformly approximated on $K$ by regular maps $K \to Y$ if and only if $f$ is homotopic to a regular map $K \to Y$. However, it can happen that a null homotopic holomorphic map $K \to Y$ does not admit uniform approximation on $K$ by regular maps $X \to Y$. Here, a map $\varphi \colon K \to Y$ is called holomorphic (resp. regular) if there exist an open (resp. a Zariski open) neighborhood $U \subseteq X$ of $K$ and a holomorphic (resp. regular) map $\tilde \varphi \colon U \to Y$ such that $\tilde\varphi|_K = \varphi$.
\end{abstract}

\hypersetup{pdfkeywords={\thekeywords}}

\section{Introduction}\label{sec:1}

Throughout this paper, \emph{algebraic varieties} are complex algebraic varieties understood in the sense of Serre \cite[p.~226]{bib25}. Each algebraic variety has an underlying structure of a complex space. Nonsingular algebraic varieties are complex (holomorphic) manifolds and will be called \emph{algebraic manifolds}. Morphisms of algebraic varieties will be called \emph{regular maps} (clearly, they are also holomorphic maps). Unless explicitly stated otherwise, all topological notions relating to algebraic varieties will refer to the Euclidean topology determined by the standard metric on $\CB$.

An algebraic variety $Y$ is said to be \emph{homogeneous} for an algebraic group $G$ if $G$ acts transitively on $Y$, the action $G \times Y \to Y$, $(a,y) \mapsto a \cdot y$ being a regular map. Note that each homogeneous algebraic variety is an equidimensional algebraic manifold. An algebraic group is said to be \emph{linear} if it is biregularly isomorphic to a Zariski closed subgroup of the general linear group $\GL_n(\CB)$, for some $n$.

Let $X$, $Y$ be algebraic varieties and let $K$ be a compact subset of $X$. A map $f \colon K \to Y$ is said to be \emph{holomorphic} (resp. \emph{regular}) if it is the restriction of a holomorphic (resp. regular) map $\tilde f \colon U \to Y$ defined on an open (resp. a Zariski open) neighborhood $U \subseteq X$ of $K$. We say that a holomorphic map $f \colon K \to Y$ can be \emph{approximated by regular maps from $K$ into $Y$} if for every neighborhood $
\UC \subseteq \C(K,Y)$ of $f$, where $\C(K,Y)$ is the space of all continuous maps endowed with the compact-open topology, there exists a regular map $K \to Y$ that belongs to $\UC$. The compact-open topology on $\C(K,Y)$ is the same as the uniform convergence topology with respect to a metric $d$ on $Y$ which induces the Euclidean topology on $Y$. Thus, a holomorphic map $f \colon K \to Y$ can be approximated by regular maps from $K$ into $Y$ if and only if for every $\varepsilon > 0$ there exists a \emph{rational} map $\varphi$ from $X$ into $Y$, with domain of definition containing $K$, such that $d(f(x), \varphi(x)) < \varepsilon$ for all $x \in K$.

Recall that a compact subset $K$ of a reduced complex space $Z$ is \emph{holomorphically convex} in $Z$ if for every point $p \in Z \setminus K$ there exists a holomorphic function $h \colon Z \to \CB$ with $|h(p)| > |h(x)|$ for all $x \in K$. In particular, every compact geometrically convex set in $\CB^n$ is holomorphically convex.

The main result of the present paper is the following.

\begin{theorem}\label{th-1-1}
Let $X$ be an affine algebraic manifold, $K$ a compact holomorphically convex set in $X$, and $Y$ a homogeneous algebraic manifold for some linear algebraic group. Then, for a holomorphic map $f \colon K \to Y$, the following conditions are equivalent:
\begin{conditions}
\item\label{th-1-1-a} $f$ can be approximated by regular maps from $K$ into $Y$.

\item\label{th-1-1-b} $f$ is homotopic to a regular map from $K$ into $Y$.
\end{conditions}
\end{theorem}

As an immediate consequence of Theorem \ref{th-1-1} we get the following.

\begin{corollary}\label{cor-1-2}
For $X$, $K$, $Y$ as in Theorem \ref{th-1-1}, every null homotopic holomorphic map from $K$ into $Y$ can be approximated by regular maps from $K$ into $Y$.
\end{corollary}

Since every continuous map defined on a geometrically convex subset of $\CB^n$ is null homotopic, we also get the following.

\begin{corollary}\label{cor-1-3}
Let $K$ be a compact geometrically convex set in $\CB^n$ and let $Y$ be a homogeneous algebraic manifold for some linear algebraic group. Then every holomorphic map from $K$ into $Y$ can be approximated by regular maps from $K$ into $Y$.
\end{corollary}

The proof of Theorem \ref{th-1-1} is given in Section \ref{sec:2}, where we first develop technical tools among which is the concept of \emph{cascade} inspired by Gromov's key notation of \emph{spray} \cite{bib16}. Theorems \ref{th-2-7} and \ref{th-2-9} are of independent interest. In Example \ref{ex-1-5} we discuss relationships between our results and previous relevant results due to Forstneri\v{c} \cite{bib11,bib12}. Subsequently we give two other illustrative examples. In the remainder of this section we assume that Theorem \ref{th-1-1} holds.

Let $f \colon K \to \CB$ be a holomorphic function defined on a compact subset $K$ of $\CB$. By the Runge approximation theorem, for every $\varepsilon > 0$ there exists a rational function $\varphi$ on~$\CB$ without poles in $K$ such that $|f(x)-\varphi(x)|<\varepsilon$ for all $x \in K$. Now suppose that the compact set $K$
 is holomorphically convex in $\CB$ (equivalently, the set $\CB \setminus K$ is connected). Then, according to another variant of the Runge approximation theorem, for every $\varepsilon > 0$ there exists a regular (= polynomial) function $\psi \colon \CB \to \CB$ such that $|f(x)-\psi(x)|<\varepsilon$ for all $x \in K$.
 
 These two versions of the Runge approximation theorem suggest two different general problems. Given two algebraic manifolds $X,Y$ and a holomorphic map $f \colon K \to Y$ defined on a compact subset $K$ of $X$, consider the following.
 
 \begin{problem}\label{prob-i}
 Under what assumptions can $f$ be uniformly approximated on $K$ by regular maps from $K$ into $Y$?
 \end{problem}
 
 \begin{problem}\label{prob-ii}
 Under what assumptions can $f$ be uniformly approximated on $K$ by regular maps from $X$ into $Y$?
 \end{problem}
 
 The case where $Y = \CB$ (or, equivalently, $Y=\CB^n$) is classical, see the excellent recent survey by Fornaes, Forstneri\v{c} and Wald \cite{bib10} and the references therein. Some linearization methods are used if $Y$ is not a vector space. Assuming algebraic subellipticity of $Y$ (see \cite[Definition 2.1]{bib11} or \cite[Definition 5.6.13(e)]{bib12}), Forstneri\v{c} \cite{bib11,bib12,bib13}, L\'arusson and Truong \cite{bib23}, and Kusakabe \cite{bib22} obtained interesting results concerning Problem \ref{prob-ii}. On the other hand, \cite{bib5, bib21} and the present paper are contributions addressing Problem \ref{prob-i}. Suppose
now that $\dim X=1$ and $K$ is an arbitrary compact subset of $X$. The recent result of Benoist and
Wittenberg \cite[Theorem D]{bib2a} implies that approximation as in Problem \ref{prob-i} (resp. Problem~\ref{prob-ii}) is always
possible if $Y$ is a homogeneous space for some linear algebraic group (resp. a nonsingular
compactification of such a homogeneous space). Moreover, by \cite[Examples~5.4]{bib2a}, approximation
as in Problem \ref{prob-ii} is always possible if $Y$ is a nonsingular projective cubic hypersurface of dimension
at least $2$ (note that the analogous statement is false for some nonsingular quartic surfaces in
projective $3$-space). Thus, in particular, Theorem \ref{th-1-1} is of interest only for $\dim X \geq 2$.
 
 The following serves to fix some notation.
 
\begin{example}\label{ex-1-4}
Here are some homogeneous algebraic manifolds for linear algebraic groups.
\begin{exconditions}[widest=iii]
\item\label{ex-1-4-i} Every linear algebraic group $G$ is a homogeneous manifold for $G$ under action by left translations.

\item\label{ex-1-4-ii} If $G$ is a linear algebraic group and $H$ is a Zariski closed subgroup of $G$, then the quotient $G/H$ is a homogeneous algebraic manifold for $G$. Each homogeneous algebraic manifold for $G$ is, up to biregular isomorphism, of the form $G/H$.

\item\label{ex-1-4-iii} The Grassmannian $\G(k,n)$ of $k$-dimensional vector subspaces of $\CB^n$ is a homogeneous algebraic manifold for $\GL_n(\CB)$. In particular, complex projective $n$-space $\PB^n = \G(1,n+1)$ is homoegeneous for $\GL_{n+1}(\CB)$.

\item\label{ex-1-4-iv} For every nonnegative integer $n$ the complex unit $n$-sphere
\begin{equation*}
    \Sigma^n = \{(z_0, \ldots, z_n) \in \CB^{n+1} : z_0^2 + \cdots + z_n^2 = 1 \}
\end{equation*}
is a homogeneous algebraic manifold for the complex special orthogonal group $\SO_{n+1}(\CB)$. The set $\Sigma^n(\R) \coloneqq \Sigma^n \cap \R^{n+1}$ of real points of $\Sigma^n$ is the real unit $n$-sphere $\SB^n$ in $\R^{n+1}$. Note that $\SB^n$ is a deformation retract of $\Sigma^n$.
\end{exconditions}
\end{example}

Next we discuss relationships between Problem~\ref{prob-i} and Problem~\ref{prob-ii} in the context of Corollary~\ref{cor-1-3} and \cite[Corollary~6.15.2]{bib12}.

\begin{example}\label{ex-1-5}
Let $K$ be a compact geometrically convex set in $X \coloneqq \CB^n$ and let $Y$ be an algebraic manifold. By Corollary~\ref{cor-1-3}, if $Y$ is homogeneous for some linear algebraic group, then every holomorphic map form $K$ into $Y$ can be uniformly approximated on~$K$ by regular maps from $K$ into $Y$. On the other hand, by \cite[Corollary~6.15.2]{bib12}, if $Y$ is algebraically subelliptic, then every holomorphic map from $K$ into $Y$ can be uniformly approximated on $K$ by regular maps from $X$ into $Y$.

Suppose that $Y$ is homogeneous for a linear algebraic group $G$. Recall that a character of~$G$ is a homomorphism of algebraic groups $G \to \CB^{\times}$, where $\CB^{\times} \coloneqq \CB \setminus \{0\}$. If $G$ is connected and without nontrivial characters, then $Y$ is algebraically flexible \cite[Proposition~5.4]{bib1}, hence algebraically elliptic (therefore also algebraically subelliptic) \cite[Proposition~5.6.22(c)]{bib12}. Now, let $H$ be the isotropy group of a point $y \in Y$ and let us identify $Y$ with the quotient $G/H$. Assume that $G$ admits a nontrivial character $\chi \colon G \to \CB^{\times}$ with $\chi(H)=\{1\}$. Note that the regular map $\varphi \colon G/H \to \CB^{\times}$, defined by $\varphi(aH)=\chi(a)$ for all $a \in G$, is surjective. Therefore we can choose a holomorphic map $f \colon K \to G/H$ such that the composite map $\varphi \circ f \colon K \to \CB^{\times}$ is nonconstant. We claim that $f$ cannot be uniformly approximated on~$K$ by regular maps from $X$ into $G/H$. Indeed, supposing the claim false, we conclude that the nonconstant holomorphic map $\varphi \circ f \colon K \to \CB^{\times}$ can be uniformly approximated on $K$ by regular maps from $X$ into $\CB^{\times}$. This yields a contradiction because every regular map $X \to \CB^{\times}$ is constant. It follows that the homogeneous algebraic manifold $G/H$ is not algebraically subelliptic. Thus, Problems \ref{prob-i} and \ref{prob-ii} lead to quite different results, and the methods of \cite{bib11, bib12, bib13, bib16, bib20, bib22, bib23} based on algebraic subellipticity are not always applicable for maps into homogeneous algebraic manifolds.

As noted in Example~\ref{ex-1-4}\ref{ex-1-4-iv}, the complex $n$-sphere $\Sigma^n$ is homogeneous for the group $\SO_{n+1}(\CB)$, which is connected and has no nontrivial character. Consequently, $\Sigma^n$ is algebraically elliptic.

By \cite[Proposition~6.4.1(a)]{bib12}, if $A$ is an algebraic subset of $\CB^n$ of dimension at most ${n-2}$, then the complement $\CB^n \setminus A$ is algebraically elliptic. Such complements are homogeneous only in exceptional cases.

In conclusion, the notions of algebraic subellipticity and algebraic homogeneity are complementary: neither implies the other.
\end{example}

The next example requires some preparation. Let $X$ be a quasiprojective algebraic manifold. For any nonnegative integer $k$, a cohomology class in $H^{2k}(X; \Z)$ is said to be \emph{algebraic} if it corresponds via the cycle map to an algebraic cycle of codimension~$k$ on $X$, see \cite[Chapter~19]{bib13a}. The set $\Halg^{2k}(X;\Z)$ of all algebraic cohomology classes in $H^{2k}(X;\Z)$ forms a subgroup. The groups $\Halg^{2k}(-;\Z)$ have the expected functorial property: If $f \colon X \to Y$ is a regular map between quasiprojective algebraic manifolds, then
\begin{equation*}
    f^*(\Halg^{2k}(Y;\Z)) \subseteq \Halg^{2k}(X;\Z),
\end{equation*}
where $f^* \colon H^{2k}(Y; \Z) \to H^{2k}(X; \Z)$ is the homomorphism induced by $f$.

Now, fix a positive integer $k$. If $U$ is a Zariski open subset of $\CB^m$, then $\Halg^{2k}(U;\Z)=0$ (it is sufficient to note that each algebraic cycle on $U$ is the restriction of an algebraic cycle on $\CB^m$, and $H^{2k}(\CB^m;\Z) =0$). Therefore, given a regular map $\varphi \colon K \to Y$ from a compact subset $K$ of $\CB^m$ into a quasiprojective algebraic manifold $Y$, we get
\begin{equation*}
    \varphi^*(\Halg^{2k}(Y;\Z)) = 0 \quad \text{in } H^{2k}(K;\Z).
\end{equation*}
This is the case since $\varphi$ is the restriction of a regular map $\psi \colon U \to Y$ defined on a Zariski open neighborhood $U \subseteq \CB^m$ of $K$, and
\begin{equation*}
    \psi^*(\Halg^{2k}(Y;\Z)) \subseteq \Halg^{2k}(U;\Z) = 0,
\end{equation*}
where the inclusion holds by the functorial property of $\Halg^{2k}(-;\Z)$.

\begin{example}\label{ex-1-6}
Let $m,n,p$ be integers satisfying $1 \leq m \leq 2n-1 < p$ and let $r_j,R_j$ be real numbers with $0<r_j<R_j$ for $j=1,\ldots,m$. The annulus
\begin{equation*}
    K_j \coloneqq \{z \in \CB : r_j \leq |z| \leq R_j\}
\end{equation*}
is a compact holomorphically convex subset of $\CB^{\times} = \CB\setminus\{0\}$, and hence the Cartesian product $K \coloneqq K_1 \times \cdots \times K_m$ is a compact holomorphically convex subset of the $m$-fold product $X \coloneqq (\CB^{\times})^m$. Clearly, $X$ is an affine algebraic manifold. We claim that each regular map $\varphi \colon K \to \G(n,p)$ into the Grassmannian $\G(n,p)$ is null homotopic. The proof depends on some topological constructions. Let $U(n,p)$ denote the tautological vector bundle over $\G(n,p)$. To any continuous map $h \colon K \to \G(n,p)$ one can assign the pullback vector bundle $h^* U(n,p)$ over $K$. This gives rise to a map
\begin{equation*}
    \sigma \colon [K, \G(n,p)] \to \tilde K_{\CB}(K)
\end{equation*}
from the set $[K,\G(n,p)]$ of homotopy classes of continuous maps $K \to \G(n,p)$ into the group $\tilde K_{\CB}(K)$ of stable equivalence classes of topological $\CB$-vector bundles over $K$. Since $K$ has the homotopy type of the $m$-dimensional torus $(\SB^1)^m$, it follows from \cite[Chap.~8, Theorems 2.6 and 4.2]{bib19} that the map $\sigma$ is bijective (the inequalities $1 \leq m \leq m \leq 2n-1 <p$ are needed here). Now, the map $\varphi \colon K \to \G(n,p)$ being regular, in view of the discussion preceding Example~\ref{ex-1-6}, we get
\begin{equation*}
    \varphi^*(H^{2k}(\G(n,p));\Z)=0 \quad \text{in } H^{2k}(K;\Z)
\end{equation*}
for all positive integers $k$ (note that $\Halg^{2k}(\G(n,p);\Z) = H^{2k}(\G(n,p);\Z)$). Therefore, for the pullback vector bundle $\varphi^*U(n,p)$ over $K$, we get $c_k(\varphi^*U(n,p)) = \varphi^*(c_k(U(n,p)))=0$, where $c_k(-)$ stands for the $k$th Chern class. By \cite[\S2.5]{bib2}, the vector bundle $\varphi^*U(n,p)$ is topologically stably trivial. We conclude that the map $
\varphi$ is null homotopic, as claimed.

We know precisely the size of the set $[K, \G(n,p)]$ of homotopy classes. Indeed, $\sigma$ is a bijection and, by \cite[\S2.5]{bib2}, $\tilde K_{\CB}(K)$ is a free Abelian group of rank equal to the rank of the direct sum $\bigoplus_{k>0} H^{2k}(K;\Z)$. In particular, the set $[K, \G(n,p)]$ is infinite if $m \geq 2$. Finally, let us note that each continuous map $K \to \G(n,p)$ is homotopic to the restriction of a holomorphic map $X\to \G(n,p)$. This is the case since $K$ is a retract of $X$, and each continuous map $X \to \G(n,p)$ is homotopic to a holomorphic map by Grauert's theorem \cite{bib14} (see \cite{bib16} and \cite[Thoerem~5.4.4]{bib12} for more general results).
\end{example}

Given an affine (complex) algebraic variety $X$ defined over $\R$, we denote by $X(\R)$ the set of real points of $X$. Note that each compact subset of $X(\R)$ is holomorphically convex in $X$. Indeed, we may assume that $X$ is an algebraic subset of $\CB^m$, for some $m$, defined by polynomial equations with real coefficients. Then $X(\R) = X \cap \R^m$ is an algebraic subset of $\R^m$. It is well-known that every compact subset of $\R^m$ is holomorphically convex in~$\CB^m$. Consequently, each compact subset of $X(\R)$ is holomorphically convex in $\CB^m$, hence it is also holomorphically convex in $X$.

Let $X$ be an affine (complex) algebraic manifold defined over $\R$, $K$ a compact subset of $X(\R)$, and $Y$ an algebraic manifold. We may regard both $X(\R)$ and $Y$ as real analytic manifolds. Clearly, a map $f \colon K \to Y$ is holomorphic if and only if there exists a real analytic map $\varphi \colon W \to Y$ defined on an open neighborhood $W \subseteq X(\R)$ of $K$ such that $\varphi|_K=f$. By \cite{bib15}, $Y$ admits a real analytic embedding in some real Euclidean space (this is straightforward and does not require \cite{bib15} if $Y$ is quasiprojective). Hence, according to \cite[Theorem~2]{bib26}, each continuous map from $W$ into $Y$ can be uniformly approximated on $K$ by real analytic maps from $W$ into $Y$. Consequently, each continuous map from $K$ into $Y$ can be uniformly approximated on $K$ by holomorphic maps from $K$ into $Y$.

In our last example below we refer to real algebraic sets and real regular maps (see \cite{bib3} for a detailed treatment of these notions).

\begin{example}\label{ex-1-7}
Let $M$ be a compact connected $\Cinfty$ manifold of dimension $n$. We assert that there exists an affine (complex) algebraic manifold $X$ defined over $\R$ such that its real part $K=X(\R)$ is diffeomorphic to $M$ and every continuous map from $K$ into the complex unit $n$-sphere $\Sigma^n$ can be approximated by (complex) regular maps from $K$ into~$\Sigma^n$. This can be proved as follows. By \cite[Proposition~4.5]{bib4}, there exists a nonsingular real algebraic set $K$ in $\R^m$, for some $m$, such that $K$ is diffeomorphic to $M$ and each continuous map from $K$ into $\SB^n$ is homotopic to a (real) regular map from $K$ into $\SB^n$. Let $Z$ be the Zariski closure of $K$ in $\CB^m$. The singular locus $S$ of $Z$ is an algebraic subset of $\CB^m$ defined by polynomial equations $P_1=0, \ldots, P_k=0$, where the polynomials $P_i$ have real coefficients. Setting $P \coloneqq P_1^2 + \cdots + P_k^2$, we get
\begin{equation*}
    S(\R) = \{x\in \R^m : P(x) = 0\}.
\end{equation*}
By construction, $K=Z(\R)$ is disjoint from $S(\R)$, and hence
\begin{equation*}
    X \coloneqq Z \setminus \{x \in \CB^m : P(x)=0\}
\end{equation*}
is an affine (complex) algebraic manifold defined over $\R$, with $X(\R) = K$. Let $f \colon K \to \Sigma^n$ be a continuous map. Our goal is to prove that $f$ can be uniformly approximated on~$K$ by regular maps. Let $j \colon \SB^n \hookrightarrow \Sigma^n$ be the inclusion map and let $\rho \colon \Sigma^n \to \SB^n$ be a deformation retraction. The composite $\rho \circ f \colon K \to \SB^n$ is homotopic to a (real) regular map $h \colon K \to \SB^n$. Setting $g = j \circ h$, we have $\rho \circ g = h$, and hence the maps $\rho \circ f, \rho \circ g \colon K \to \SB^n$ are homotopic. Consequently, the maps $f,g \colon K \to \Sigma^n$ are also homotopic. Since $f$~is uniformly approximable by holomorphic maps from $K$ into $\Sigma^n$, we may assume that $f$~itself is a holomorphic map. In view of Theorem~\ref{th-1-1} the proof is complete.
\end{example}

Regular maps are in general too rigid for approximation of holomorphic maps. Nash maps, which form an intermediate class between regular and holomorphic maps, are more flexible. Demailly, Lempert and Shiffman \cite{bib9} and Lempert \cite{bib24} proved that a holomorphic map from a Runge domain in an affine variety into a quasiprojective variety can be uniformly approximated on compact sets by Nash maps.

\section{Sections of amenable submersions}\label{sec:2}
In this section we work with vector bundles which are always either holomorphic or algebraic vector bundles. Let $Y$ be a complex (holomorphic) manifold or an algebraic manifold. Given a vector bundle $p \colon E \to Y$ over $Y$, with total space $E$ and bundle projection $p$, we may refer to $E$ as a vector bundle over $Y$. If $y \in Y$, we let $E_y \coloneqq p^{-1}(y)$ denote the fiber of $E$ over $y$ and write $0_y$ for the zero vector in $E_y$. The set $Z(E) = \{0_y : y \in Y \}$ is called the zero section of $E$. In particular, if $E = Y \times \CB^n$ is the product vector bundle over $Y$, then $E_y = \{y\} \times \CB^n$, $0_y = (y,0)$, and $Z(E) = Y \times \{0\}$ (here $0$ is the zero vector in $\CB^n$). We write $TY$ for the tangent bundle to $Y$ and $T_yY$ for the tangent space to $Y$ at $y \in Y$.

Next we introduce some notations and definitions, and prove three technical lemmas.

\begin{notation}\label{not-2-1}
Let $X, Z$ be two algebraic manifolds and let $h \colon Z \to X$ be a regular map which is a surjective submersion. Furthermore, let $V(h)$ denote the algebraic vector subbundle of the tangent bundle $TZ$ to $Z$ defined by
\begin{equation*}
    V(h)_z = \Ker (d_zh \colon T_zZ \to T_{h(z)}X) \quad \text{for all } z \in Z,
\end{equation*}
where $d_zh$ is the derivative of $h$ at $z$. Clearly, $V(h)_z$ is the tangent space to the fiber $h^{-1}(h(z))$.
\end{notation}

For any subset $A$ of $X$, a map $\alpha \colon A \to Z$ is called a \emph{section} of $h \colon Z \to X$ if $h(\alpha(x)) = x$ for all $x \in A$. Given an open subset $U$ of $X$, we call a continuous map ${F \colon U \times [0,1] \to Z}$ a~\emph{homotopy of holomorphic sections} if for every $t \in [0,1]$ the map $F_t \colon U \to Z$, ${x \mapsto F(x,t)}$ is a holomorphic section. Now, let $K$ be a compact subset of $X$. A section ${f \colon K \to Z}$ is said to be \emph{holomorphic} (resp. \emph{regular}) if it is the restriction of a holomorphic (resp. regular) section $\tilde f \colon U \to Z$ defined on an open (resp. Zariski open) neighborhood $U \subseteq X$ of $K$. We say that a holomorphic section $f \colon K \to Z$ can be \emph{approximated by regular sections defined on $K$} if for every neighborhood $\UC$ of $f$ in the space $\C(K,Z)$ of all continuous maps there exists a regular section $\varphi \colon K \to Z$ that belongs to $\UC$. Two holomorphic sections $f_0, f_1 \colon K \to Z$ are said to be \emph{homotopic through holomorphic sections} if there exist an open neighborhood $U \subseteq X$ of $K$ and a homotopy of holomorphic sections $F \colon U \times [0,1] \to Z$ such that $F_0|_K = f_0$ and $F_1|_K = f_1$.\goodbreak

\begin{definition}\label{def-2-2}
Let $h \colon Z \to X$ be the submersion of Notation \ref{not-2-1}.
\begin{iconditions}
\item\label{def-2-2-i} A \emph{cascade} for $h \colon Z \to X$ is a triple $(E,E^0,s)$, where $E = Z \times \CB^n$ is the product vector bundle over $Z$, for some $n$, and $s \colon E^0 \to Z$ is a regular map defined on a Zariski open neighborhood $E^0 \subseteq E$ of the zero section $Z(E) = Z \times \{0\}$ of $E$ such that
\begin{equation*}
    s(E_z \cap E^0) \subseteq h^{-1}(h(z)) \quad\text{and}\quad s(z,0) = z \quad\text{for all } z \in Z.
\end{equation*}

\item\label{def-2-2-ii} A cascade $(E, E^0, s)$ for $h \colon Z \to X$ is said to be \emph{dominating} if the derivative
%--displayed because of a bad break
\begin{equation*}
    d_{(z,0)}s \colon T_{(z,0)}E \to T_zZ
\end{equation*}
maps the subspace $E_z = T_{(z,0)}E_z$ of $T_{(z,0)}E$ onto $V(h)_z$, that is,
\begin{equation*}
    d_{(z,0)}s(E_z) = V(h)_z
\end{equation*}
for all $z\in Z$.

\item\label{def-2-2-iii} The submersion $h \colon Z \to X$ is called \emph{amenable} if it admits a dominating cascade.
\end{iconditions}
\end{definition}

\begin{notation}\label{not-2-3}
Suppose that $(E=Z \times \CB^n, E^0, s)$ is a dominating cascade for the submersion $h \colon Z \to X$. Let $f \colon U \to Z$ be a holomorphic section of $h \colon Z \to X$ defined on some open subset $U$ of $X$. Let $E_f = U \times \CB^n$ be the product vector bundle over $U$ and define
\begin{align*}
    &E_f^0 \coloneqq \{(x,v) \in U \times \CB^n : (f(x), v) \in E^0 \}\\
    &s_f \colon E_f^0 \to Z, \quad s_f(x,v) \coloneqq s(f(x),v).
\end{align*}
Clearly, $E_f^0 \subseteq E_f$ is an open neighborhood of the zero section $Z(E_f) = U \times \{0\}$ of $E_f$, and $s_f$ is a holomorphic map.
\end{notation}

\begin{lemma}\label{lem-2-4}
With Notation~\ref{not-2-3}, assume that the open subset $U$ of $X$ is Stein. Then there exists a holomorphic vector subbundle $\tilde E_f$ of $E_f$ having the following property: If $\tilde E_f^0 \coloneqq \tilde E \cap E_f^0$ and $\tilde s_f \colon \tilde E_f^0 \to Z$ is the restriction of $s_f$, then $\tilde s_f$ maps biholomorphically an open neighborhood of the zero section $Z(\tilde E_f) = U \times \{0\}$ in $\tilde E_f^0$ onto an open neighborhood of $f(U)$ in $Z$.
\end{lemma}

\begin{proof}
For an arbitrary point $x \in U$ the zero vector in the fiber $(E_f)_x = \{x\} \times \CB^n$ is $(x,0)$. Moreover, the derivative
\begin{equation*}
    d_{(x,0)}s_f \colon T_{(x,0)}E_f \to T_{f(x)}Z
\end{equation*}
induces a surjective linear map
\begin{equation*}
    \varphi_x \colon (E_f)_x \to V(h)_{f(x)}
\end{equation*}
(we regard $(E_f)_x = T_{(x,0)}(E_f)_x$ as a vector subspace of $T_{(x,0)}E_f$). Note that the union $K \coloneqq \bigcup_{x\in U} \Ker \varphi_x$ is a holomorphic vector subbundle of $E_f$. Since $U$ is a Stein open subset of $X$, it follows from Cartan's theorem B that $E_f$ can be expressed as a direct sum $E_f = \tilde E_f \oplus K$ for some holomorphic vector subbundle $\tilde E_f$ of $E_f$, see \cite[Corollary~2.6.6]{bib12}. Set $\tilde E_f^0 \coloneqq \tilde E_f \cap E_f^0$ and let $\tilde s_f \colon \tilde E_f^0 \to Z$ be the restriction of $s_f \colon E_f \to Z$. We have $\tilde s_f(x,0) = s(f(x), 0)=f(x)$ for all $x \in U$, hence $\tilde s_f$ induces a biholomorphism between $Z(\tilde E_f) = U \times \{0\}$ and $f(U)$. Moreover, by construction, the derivative
\begin{equation*}
    d_{(x,0)} \tilde s_f \colon T_{(x,0)} \tilde E_f \to T_{f(x)} Z
\end{equation*}
is an isomorphism for all $x \in U$. Consequently, $\tilde s_f$ is a local biholomorphism at each point $(x,0)$. Therefore the lemma follows from \cite[(12.7)]{bib7}.
\end{proof}

\begin{lemma}\label{lem-2-5}
Suppose that $(E=Z \times \CB^n, E^0, s)$ is a dominating cascade for the submersion $h \colon Z \to X$. Let $U$ be an open Stein subset of $X$ and let $F \colon U \times [0,1] \to Z$ be a homotopy of holomorphic sections of $h \colon Z \to X$. Let $U_0$ be an open subset of $X$ whose closure $\overline U_0$ is compact and contained in $U$. Let $t_0 \in [0,1]$. Then there exist a neighborhood $I_0$ of $t_0$ in~$[0,1]$ and a continuous map $\eta \colon U_0 \times I_0 \to \CB^n$ such that
\begin{inthm}[widest=2.5.4]
\item\label{lem-2-5-1} $(F(x,t_0), \eta(x,t)) \in E^0$ for all $(x,t) \in U_0 \times I_0$,

\item\label{lem-2-5-2} $\eta(x,t_0) = 0$ for all $x\in U_0$,

\item\label{lem-2-5-3} $s(F(x,t_0), \eta(x,t)) = F(x,t)$ for all $(x,t) \in  U_0 \times I_0$,

\item\label{lem-2-5-4} for every $t \in I_0$ the map $U_0 \to \CB^n$, $x \mapsto \eta(x,t)$ is holomorphic.
\end{inthm}
\end{lemma}

\begin{proof}
Define a holomorphic section $f \colon U \to Z$ by $f(x) = F(x,t_0)$. By Lemma \ref{lem-2-4}, there exists a holomorphic subbundle $\tilde E_f$ of $E_f$ such that if $\tilde E_f^0 = \tilde E_f \cap E^0$ and $\tilde s_f \colon \tilde E_f^0 \to Z$ is the restriction of $s_f$, then $\tilde s_f$ maps biholomorphically an open neighborhood $M \subseteq \tilde E_f^0$ of the zero section $Z(\tilde E_f) = U\times \{0\}$ onto an open neighborhood $N \subseteq Z$ of $f(U)$. Let $\sigma \colon M \to N$ be the restriction of $\tilde s_f$. Since $\overline U_0$ is a compact subset of $U$, we can choose an open neighborhood $I_0$ of $t_0$ in $[0,1]$ such that $F_t(U_0) \subseteq N$ for all $t \in I_0$. Therefore, for every $t \in I_0$, there exists a unique holomorphic map $\xi_t \colon U_0 \to M$ satisfying $F_t(x) = \sigma(\xi_t(x))$ for all $x \in U_0$. Writing $\xi_t(x)$ as $\xi_t(x) = (\alpha_t(x), \eta_t(x))$, where $\alpha_t \colon U_0 \to U$ and $\eta_t \colon U_0 \to \CB^n$ are holomorphic maps, for all $(x,t) \in U_0 \times I_0$ we get
\begin{equation*}
    F_t(x) = \tilde s_f(\alpha_t(x), \eta_t(x)) = s(f(\alpha_t(x)), \eta_t(x)),
\end{equation*}
hence also
\begin{equation*}
    x=h(F_t(x)) = h(s(f(\alpha_t(x)), \eta_t(x))) = \alpha_t(x),
\end{equation*}
where the last equality follows from Definition \ref{def-2-2}\ref{def-2-2-i}. Consequently, $\alpha_t(x) = x$ for all $(x,t) \in U_0 \times I_0$. By construction, $\eta \colon U_0 \times I_0 \to \CB^n$, $(x,t) \mapsto \eta_t(x)$ is a continuous map and conditions \ref{lem-2-5-1}--\ref{lem-2-5-4} hold.
\end{proof}

\begin{lemma}\label{lem-2-6}
Suppose that the submersion $h \colon Z \to X$ is amenable. Let $U$ be an open Stein subset of $X$ and let $F \colon U \times [0,1] \to Z$ be a homotopy of holomorphic sections of $h \colon Z \to X$. Let $U_0$ be an open subset of $X$ whose closure $\overline U_0$ is compact and contained in $U$. Then there exists a dominating cascade $(E = Z \times \CB^m, E^0, s)$ for $h \colon Z \to X$ and a continuous map $\eta \colon U_0 \times [0,1] \to \CB^m$ such that
\begin{inthm}
\item\label{lem-2-6-1} $(F(x,0),\eta(x,t)) \in E^0$ for all $(x,t) \in U_0 \times [0,1]$,

\item\label{lem-2-6-2} $\eta(x,0) = 0$ for all $x \in U_0$,

\item\label{lem-2-6-3} $s(F(x,0),\eta(x,t)) = F(x,t)$ for all $(x,t) 
\in U_0 \times [0,1]$,

\item\label{lem-2-6-4} for every $t \in [0,1]$ the map $U_0 \to \CB^m$, $x \mapsto \eta(x,t)$ is holomorphic.
\end{inthm}
\end{lemma}

\begin{proof}
Let $(\tilde E= Z \times \CB^n, \tilde E^0, \tilde s)$ be a dominating cascade for the submersion $h\colon Z \to X$. In view of Lemma \ref{lem-2-5} and the compactness of the interval $[0,1]$ (see the Lebesgue lemma for compact metric spaces \cite[p.~28, Lemma~9.11]{bib6}), there exists a partition $0 = t_0 < t_1 < \cdots < t_k = 1$ of $[0,1]$ such that for $i=1,\ldots,k$ there exists a continuous map $\eta^i \colon U_0 \times [t_{i-1},t_i] \to \CB^n$ with the following properties:
\begin{itemize}
\begin{samepage}
    \item $F(x,t_{i-1}), \eta^i(x,t)) \in \tilde E^0$ for all $(x,t) \in U_0 \times [t_{i-1},t_1]$,
    \item $\eta^i(x,t_{i-1})=0$ for all $x \in U_0$,
\end{samepage}
    \item $\tilde s(F(x,t_{i-1}), \eta^i(x,t)) = F(x,t)$ for all $(x,t) \in U_0 \times [t_{i-1},t_i]$,
    \item for every $t \in [t_{i-1},t_i]$ the map $U_0 \to \CB^n$, $x \mapsto \eta^i(x,t)$ is holomorphic.
\end{itemize}
For $i=1,\ldots,k$ we define recursively a dominating cascade $(E(i), E(i)^0, s^{(i)})$  by
\begin{equation*}
    (E(i), E(i)^0, s^{(i)}) \coloneqq (\tilde E, \tilde E^0, \tilde s) \quad \text{if } i=1,
\end{equation*}
while for $i \geq 2$ we set
\begin{align*}
    &E(i) \coloneqq Z \times (\CB^n)^{(i)}\\
    &E(i)^0 \coloneqq \begin{aligned}[t]
    \{(z,v_1, \ldots, v_i) \in E(i) :
    (&z, v_1, \ldots, v_{i-1}) \in E(i-1)^0, \\ 
    (&s^{i-1}(z, v_1, \ldots, v_{i-1}), v_i) \in E(1)^0\},
    \end{aligned}\\
    &s^{(i)} \colon E(i)^0 \to Z, \quad s^{(i)}(z, v_1, \ldots, v_i) \coloneqq s^{(1)}(s^{(i-1)}(z, v_1, \ldots, v_{i-1}),v_i),
\end{align*}
where $z$ is in $Z$ and $v_1, \ldots, v_i$ are in $\CB^n$.

In particular, $(E, E^0, s) \coloneqq (E(k), E(k)^0, s^{(k)})$ is a dominating cascade for $h \colon Z \to X$. By construction, $E = Z \times \CB^m$, where $\CB^m = (\CB^n)^k$ is the $k$-fold product of $\CB^n$. Now, consider a map $\eta \colon U_0 \times [0,1] \to \CB^m = (\CB^n)^k$ defined by
\begin{equation*}
    \eta(x,t) \coloneqq (\eta^1(x,t), 0, \ldots, 0)
\end{equation*}
for all $(x,t) \in U_0 \times [t_0, t_1]$, and
\begin{equation*}
    \eta(x,t) \coloneqq (\eta^1(x,t), \ldots, \eta^{i-1}(x,t), \eta^i(x,t), 0, \ldots, 0)
\end{equation*}
for all $(x,t) \in U_0 \times [t_{i-1},t_i]$ with $i=2,\ldots,k$. One readily checks that $\eta$ is a well-defined continuous map satisfying \ref{lem-2-6-1}--\ref{lem-2-6-4}.
\end{proof}

We have the following result on approximation of holomorphic sections by regular sections.

\begin{theorem}\label{th-2-7}
Let $h \colon Z \to X$ be an amenable regular submersion from an algebraic manifold $Z$ onto an affine algebraic manifold $X$. Let $K$ be a compact holomorphically convex subset of $X$ and let $f \colon K \to Z$ be a holomorphic section of $h \colon Z \to X$ that is homotopic through holomorphic sections to a regular section defined on $K$. Then $f$ can be approximated by regular sections defined over $K$.
\end{theorem}

\begin{proof}
By assumption, there exist an open neighborhood $U \subseteq X$ of $K$ and a homotopy $F \colon U \times [0,1] \to Z$ of holomorphic sections of $h \colon Z \to X$ such that $F_0|_K$ is a regular section and $F_1|_K=f$. Shrinking $U$ if necessary, we may assume that $U$ is Stein. Choose an open subset $U_0$ of $X$ such that its closure $\overline U_0$ is compact and $K \subseteq U_0 \subseteq \overline U_0 \subseteq U$. Let $(E= Z \times \CB^m, E^0, s)$ and $\eta \colon U_0 \times [0,1] \to \CB^m$ be as in Lemma~\ref{lem-2-6}. In particular, by~\ref{lem-2-6-3}, we get
\begin{equation*}
    s(F_0(x),\eta_1(x)) = F_1(x) \quad \text{for all } x \in U_0,
\end{equation*}
where $\eta_1 \colon U_0 \to \CB^m$, $\eta_1(x) = \eta(x,1)$.

Since $X$ is an affine algebraic manifold, we may assume that $X$ is an algebraic subset of~$\CB^N$. Hence, by \cite[p.~245, Theorem~18]{bib17}, every holomorphic function on $X$ is the restriction of a holomorphic function on $\CB^N$. It follows that the compact set $K$ is polynomially convex in $\CB^N$, being holomorphically convex in $X$. Now, represent $U_0$ as $U_0 = W \cap X$, where $W$ is an open neighborhood of $K$ in $\CB^N$. Choose a Stein open neighborhood~$W_1$ of $K$ in $W$. Then the intersection $U_1 \coloneqq U_0 \cap W_1$ is a closed complex submanifold of~$W_1$. By \cite[p.~245, Theorem~18]{bib17} once again, the holomorphic map ${\eta_1|_{U_1} \colon U_1 \to \CB^m}$ has a holomorphic extension $W_1 \to \CB^m$. Therefore, according to the Oka--Weil theorem \cite[Theorem~2.7.7]{bib18}, the map $\eta_1|_K \colon K \to \CB^m$ can be uniformly approximated on $K$ by polynomial maps $\CB^N \to \CB^m$. If $\beta \colon \CB^N \to \CB^m$ is a polynomial map with $\beta|_K$ sufficiently close to $\eta_1|_K$, then $(F_0(x), \beta(x)) \in E^0$ for all $x \in K$, and
\begin{equation*}
    \varphi \colon K \to Z, \quad x \mapsto s(F_0(x), \beta(x))
\end{equation*}
is a regular map close to $f$. By Definition~\ref{def-2-2}\ref{def-2-2-i}, $\varphi$ is a section of $h \colon Z \to X$, which completes the proof.
\end{proof}

Our next task is to derive from Theorem~\ref{th-2-7} a result on approximation of holomorphic maps. To this end the following is useful.

\begin{definition}\label{def-2-8}
Let $Y$ be an algebraic manifold.
\begin{iconditions}
\item\label{def-2-8-i} A \emph{cascade} for $Y$ is a triple $(E,E^0,s)$, where $E = Y \times \CB^n$ is the product vector bundle over $Y$, for some $n$, and $s \colon E^0 \to Y$ is a regular map defined on a Zariski open neighborhood $E^0 \subseteq E$ of the zero section $Z(E) = Y \times \{0\}$ of $E$ such that
$s(y,0)=y$ for all $y \in Y$.

\item\label{def-2-8-ii} A cascade $(E, E^0, s)$ for $Y$ is said to be \emph{dominating} if the derivative
%--display because of a bad break
\begin{equation*}
    d_{(y,0)}s \colon T_{(y,0)}E \to T_yY
\end{equation*}
maps the subspace $E_y = T_{(y,0)}E_y$ of $T_{(y,0)}E$ onto $T_yY$, that is,
\begin{equation*}
    d_{(y,0)}s(E_z) = T_yY \quad \text{for all } y \in Y.
\end{equation*}

\item\label{def-2-8-iii} The algebraic manifold $Y$ is called \emph{amenable} if it admits a dominating cascade.
\end{iconditions}
\end{definition}

Note that Definition~\ref{def-2-8} is a special case of Definition~\ref{def-2-2} for the trivial submersion $h \colon Y \to X$, where $X$ is reduced to a single point. A cascade (resp. dominating cascade) $(E,E^0,s)$ for $Y$ with $E^0=E$ is just what has been called in the literature an \emph{algebraic spray} (resp. \emph{algebraic dominating spray}) for $Y$ on the product vector bundle $E$ \cite{bib11, bib12, bib13, bib16, bib20, bib22, bib23}.

Let $X$ and $Y$ be complex (holomorphic) manifolds. Given a compact subset $K$ of $X$, we say that two holomorphic maps $f_0, f_1 \colon K \to Y$ are \emph{homotopic through holomorphic maps} if there exist an open neighborhood $U \subseteq X$ of $K$ and a continuous map $F \colon U \times [0,1] \to Y$ such that for every $t \in [0,1]$ the map $F_t \colon U \to Y$, $x \mapsto F(x,t)$ is holomorphic and $F_0|_K=f_0$, $F_1|_K=f_1$.

\begin{theorem}\label{th-2-9}
Let $X$ be an affine algebraic manifold, $K$ a compact holomorphically convex subset of $X$, and $Y$ an amenable algebraic manifold. Let $f \colon K \to Y$ be a holomorphic map that is homotopic through holomorphic maps to a regular map from $K$ into $Y$. Then $f$ can be approximated by regular maps from $K$ into $Y$.
\end{theorem}

\begin{proof}
First observe that the canonical projection $\pi \colon X \times Y \to X$ is amenable, that is, admits a dominating cascade. Indeed, by assumption, there is a dominating cascade $(E=Y \times \CB^n, E^0, s)$ for $Y$. We obtain a dominating cascade $(\tilde E, \tilde E^0, \tilde s)$ for $\pi \colon X \times Y \to X$, where
\begin{align*}
    &\tilde E \coloneqq (X \times Y) \times \CB^n,\\
    &\tilde E^0 \coloneqq \{((x,y),v) \in \tilde E : (y,v) \in E^0 \}\\
    &\tilde s \colon \tilde E^0 \to X \times Y, \quad \tilde s ((x,y),v) \coloneqq s(y,v).
\end{align*}
Clearly, $\varphi \colon K \to X \times Y$, $x \mapsto (x,f(x))$ is a holomorphic section of $\pi \colon X \times Y \to X$. By assumption once again, $\varphi$ is homotopic through holomorphic sections to a regular section defined on $K$. Hence, in view of Theorem~\ref{th-2-7}, the section $\varphi$ can be approximated by regular sections defined on $K$. Consequently, $f \colon K \to Y$ can be approximated by regular maps from $K$ into $Y$.
\end{proof}

For the proof of Theorem\ref{th-1-1} we need two additional auxiliary results.

\begin{proposition}\label{prop-2-10}
Every algebraic manifold that is homogeneous for some linear algebraic group is amenable.
\end{proposition}

\begin{proof}
Let $G$ be a linear algebraic group and let $G^0$ be the irreducible component of $G$ that contains the identity element~$1$. Set $n \coloneqq \dim G$. By a result of Chevalley \cite[Corollary~2]{bib8}, $G^0$ is a rational variety, and hence there exist a nonempty Zariski open subset $U$ of $\CB^n$ and a regular map $\psi \colon U \to G^0$ such that the image $\psi(U)$ is a Zariski open subset of $G^0$ and $\psi$ induces a biregular isomorphism from $U$ onto $\psi(U)$. Using a translation in $\CB^n$ we may assume that $U$ contains the origin $0 \in \CB^n$. Then the map $\varphi \colon U \to G$, $v\mapsto \psi(v)\psi(0)^{-1}$ is regular, $\varphi(0)=1$, and the derivative of $\varphi$ at $0$ is a linear isomorphism. Now, let $Y$ be a homogeneous algebraic manifold for $G$. We obtain a dominating cascade $(E,E^0,s)$ for~$Y$, where $E \coloneqq Y \times \CB^n$, $E^0 \coloneqq Y \times U$, and $s \colon E^0 \to Y$ is defined by $s(y,v) \coloneqq \varphi(v)\cdot y$ for all $(y,v) \in E^0$.
\end{proof}

Our last lemma holds in the framework of complex (holomorphic) manifolds.

\begin{lemma}\label{lem-2-11}
Let $X$ be a Stein manifold, $K$ a compact holomorphically convex subset of $X$, and $Y$ a complex manifold that admits a transitive action of a complex Lie group. Let $f_0, f_1 \colon K \to Y$ be two homotopic holomorphic maps. Then $f_0$ and $f_1$ are homotopic through holomorphic maps.
\end{lemma}

\begin{proof}
By assumption, there exists a continuous map $A \colon K \times [0,1] \to Y$ with $A_0=f_0$, $A_1=f_1$. The maps $f_0, f_1$ being holomorphic, we can choose an open neighborhood $V \subseteq X$ of $K$ and holomorphic maps $\tilde f_0, \tilde f_1 \colon V \to Y$ satisfying $\tilde f_0|_K=f_0$, $\tilde f_1|_K=f_1$. We claim that there exists a compact neighborhood $L \subseteq V$ of $K$ such that the maps $\tilde f_0|_L$, $\tilde f_1|_L$ are homotopic. The proof of the claim is purely topological. We may assume that $Y$ is a $\Cinfty$ submanifold of $\R^n$ for some $n$. Let $\rho \colon T \to Y$ be a tubular neighborhood of $Y$ in $\R^n$ (that is, $T \subseteq \R^n$ is an open neighborhood of $Y$ and $\rho$ is a $\Cinfty$ retraction). Picking a compact neighborhood $L \subseteq V$ of $K$, we obtain a continuous map
\begin{equation*}
    B \colon (K \times [0,1]) \cup (L \times \{0,1\}) \to \R^n
\end{equation*}
defined by $B|_{K\times[0,1]}=A$ and $B(x,0) = \tilde f_0(x)$, $B(x,1)=\tilde f_1(x)$ for all $x \in L$. Hence, by the Tietze extension theorem, there is a continuous extension $C \colon X \times [0,1] \to \R^n$ of $B$. Shrinking $L$, if necessary, we get $C(L \times [0,1])\subseteq T$. Consequently, $D \colon L \times [0,1] \to Y$, $(x,t) \mapsto \rho(C(x,t))$ is a homotopy between $\tilde f_0|_L$ and $\tilde f_1|_L$, as required.

Since the compact set $K$ is holomorphically convex in $X$, there exists an open Stein subset $U$ of $X$ with $K \subseteq U \subseteq L$. By Grauert's theorem \cite{bib14}, there is a continuous map $F \colon U \times [0,1] \to Y$ such that for every $t \in [0,1]$ the map $F_t \colon U \to Y$ is holomorphic and $F_0 = \tilde f_0|_U$, $F_1 = \tilde f_1|_U$. Hence $f_0$ and $f_1$ are homotopic through holomorphic maps.
\end{proof}

\begin{remark}\label{rem-2-12}
The conclusion of Lemma~\ref{lem-2-11} holds under a weaker assumption on the complex manifold $Y$. Namely, it is sufficient to assume that $Y$ is (holomorphically) subelliptic. In that case, in the final step of the proof, one refers to \cite[Theorem~5.4.4]{bib12} instead of \cite{bib14}.
\end{remark}

\begin{proof}[Proof of Theorem~\ref{th-1-1}]
By Proposition~\ref{prop-2-10}, the algebraic manifold $Y$ is amenable. Hence \ref{th-1-1-b} implies \ref{th-1-1-a} in view of Theorem~\ref{th-2-9} and Lemma~\ref{lem-2-11}. Obviously, \ref{th-1-1-a} implies \ref{th-1-1-b} since any two sufficiently close continuous maps in $\C(K,Y)$ are homotopic.
\end{proof}

\begin{acknowledgements}
The second named author was partially supported by the National Science Center (Poland) under grant number 2018/31/B/ST1/01059.

We thank Olivier Wittenberg for very useful comments.
\end{acknowledgements}

%\cleardoublepage
\phantomsection

\end{document}